\documentclass{article}

\usepackage{arxiv}

\usepackage[utf8]{inputenc} 
\usepackage[T1]{fontenc}    
\usepackage{hyperref}       
\usepackage{url}            
\usepackage{booktabs}       
\usepackage{amsfonts}       
\usepackage{nicefrac}       
\usepackage{microtype}      
\usepackage{graphicx}
\usepackage{amsmath}
\usepackage{amsthm}
\newtheorem{thm}{Theorem}[section]
\newtheorem{lem}[thm]{Lemma}
\newtheorem{prop}[thm]{Proposition}
\newtheorem{cor}[thm]{Corollary}
\theoremstyle{definition}
\newtheorem{exmp}{Example}[section]
\theoremstyle{remark}
\newtheorem{rem}{Remark}
\theoremstyle{remark}

\theoremstyle{definition}

\title{Finite, fiber- and orientation-preserving group actions on totally orientable Seifert manifolds}

\author{
  Benjamin Peet\\
  Department of Mathematics\\
  St. Martin's University\\
  Lacey, WA 98503 \\
  \texttt{bpeet@stmartin.edu} \\
}

\begin{document}
\maketitle

\begin{abstract}
In this paper we consider the finite groups that act fiber- and orientation-preservingly on closed, compact, and orientable Seifert manifolds that fiber over an orientable base space. We establish a method of constructing such group actions and then show that if an action satisfies a condition on the obstruction class of the Seifert manifold, it can be derived from the given construction. The obstruction condition is refined and the general structure of the finite groups that act via the construction is provided.

\end{abstract}


\section{Introduction}

\subsection{Discussion of Results}
The main question asked in this paper is: “What are the possible finite, fiber- and orientation-preserving group actions on a closed, compact, and orientable Seifert manifold with orientable base space?” We consider this by first providing a construction of an orientation-preserving group action on a given Seifert manifold. This construction is founded upon the way a Seifert manifold is put together as Dehn fillings of $S^{1}\times F$. Here $F$ is a surface with boundary. The construction is - in a general sense - to take a product action on $S^{1}\times F$ and extend across the Dehn fillings. We will refer to actions that can be constructed in this way as \textit{extended product actions}.

Any fiber-preserving group action can only exchange critical fibers if they are of the same type, so drilling and refilling these trivially will leave an action on a trivially fibered Seifert manifold. This may or may not be a product however. It is the obstruction class that determines this. 

Our main result then states:
\newtheorem*{thm:associativity}{Theorem \ref{thm:associativity}}
\begin{thm:associativity}
 Let $M$ be a closed, compact, and orientable Seifert $3$-manifold that fibers over an orientable base space. Let $\varphi:G\rightarrow Diff_{+}^{fp}(M)$ be a finite group action on $M$ such that the obstruction class can expressed as $$b=\sum_{i=1}^{m}(b_{i}\cdot\#Orb_{\varphi}(\alpha_{i}))$$ for a collection of fibers $\{\alpha_{1},\ldots,\alpha_{m}\}$ and integers $\{b_{1},\ldots,b_{m}\}$. Then $\varphi$ is an extended product action.
\end{thm:associativity}

In order to establish this result we analyze, refine, and rework Theorem 2.3 of Peter Scott and William Meeks in their paper \textit{Finite group actions on 3-manifolds} \cite{Meeks1986}. This result establishes that if a finite action on $S^{1}\times F$ respects the product structure on the boundary, then there is a product structure that agrees with the original product structure on the boundary and remains invariant under the action. This result allows us to consider when finite actions can be constructed via the given method, that is, are extended product actions.

The main result then shows that given a finite, orientation and fiber-preserving action, the action can be constructed via the given method - provided it satisfies a condition on the obstruction class of the Seifert manifold. This is within Theorem 5.3 but is specifically given by the following:

If $\varphi:G\rightarrow Diff_{+}^{fp}(M)$ is a finite group action, we will call satisfaction of $$b=\sum_{i=1}^{s}(b_{i}\cdot\#Orb_{\varphi}(\alpha_{i}))$$ for some fibers $\{\alpha_{1},\ldots,\alpha_{s}\}$ and integers $\{b_{1},\ldots,b_{s}\}$, \textit{satisfying the obstruction condition}. 

This obstruction condition will be refined and the general structure of such a group provided.

\subsection{Preliminary Definitions}
We first give some preliminary definitions.
Throughout this paper we will use $M$ to denote a closed, compact, connected, orientable (and oriented) smooth manifold of dimension $3$. $\hat{M}$ will denote a compact, orientable (and oriented) smooth manifold of dimension $3$ with boundary. $G$ will be a finite group. We let $Diff(M)$ be the group of self-diffeomorphisms of $M$, and then define a $G$-action on $M$ to be an injection $\varphi:G\rightarrow Diff(M)$. We use the notation $Diff_{+}(M)$ for the group of orientation-preserving self-diffeomorphisms of $M$. 

$M$ will further be assumed to be a \textit{Seifert-fibered} manifold. We use the original Seifert definition. That is, a Seifert manifold is a $3$-manifold such that $M$ can be decomposed into disjoint fibers where each fiber is a simple closed curve. Then for each fiber $\gamma$, there exists a fibered neighborhood (that is, a subset consisting of fibers and containing $\gamma$) which can be mapped under a fiber-preserving map onto a solid fibered torus. A fiber is known as \textit{regular} if the solid fibered torus is trivially fibered and \textit{critical} if it is not. For further details see the original work of Herbert Seifert in his dissertation \textit{Topologie Dreidimensionaler Gefaserter R{\"a}ume} \cite{Seifert1933}.

It should be noted here that due to the compactness of $M$, the number of critical fibers necessarily must be finite. For a proof of this see John Hempel's \textit{3-Manifolds} \cite{hempel}.

A \textit{Seifert bundle} is a Seifert manifold $M$ (or $\hat{M}$) along with a continuous map $p:M\rightarrow B$ where $p$ identifies each fiber to a point. Note that $B$ is an orbifold without mirror lines, but with cone points refering to the critical fibers. For clarity, we denote the underlying space of $B$ as $B_{U}$. In our case this will be a compact, orientable (and oriented) surface without boundary for $M$ and with boundary for $\hat{M}$.

Following William Thurston's \textit{The geometry and topology of 3-manifolds} \cite{thurstongeometry}, we use the notation $(n_{1},\ldots,n_{k};m_{1},\ldots,m_{l})$ as a data set for a $2$-orbifold $B$ with $k$ cone points of orders $n_{1},\ldots,n_{k}$, and $l$ corner reflectors of orders $m_{1},\ldots,m_{l}$. 

A $G$-action $\varphi$ is said to be \textit{fiber-preserving} on a Seifert manifold $M$ if for any fiber $\gamma$ and any $g\in G$, $\varphi(g)(\gamma)$ is some fiber of $M$. We use the notation $Diff^{fp}(M)$ for the group of fiber-preserving self-diffeomorphisms of $M$ (given some Seifert fibration). Given a fiber-preserving $G$-action, there is an induced action $\varphi_{B_{U}}:G_{B_{U}}\rightarrow Diff(B_{U})$ on the underlying space $B_{U}$ of the base space $B$. 

For distinction, we use the notation $Diff^{I-fp}(N)$ to refer to $I$-fiber-preserving diffeomorphisms of a manifold $N$. An $I$-fibration or a \textit{fibration by arcs} is a decomposition of the manifold $N$ into disjoint fibers each of which is diffeomorphic to the unit interval $I$.

For a finite action $\varphi:G\rightarrow Diff^{fp}(M)$, we define the \textit{orbit number of a fiber} $\gamma$ under the action to be $\#Orb_{\varphi}(\gamma)=\#\{\alpha|\varphi(g)(\gamma)=\alpha\textrm{ for some }g\in G\}$. 

If we have a manifold $\hat{M}$, then a \textit{product structure} on $\hat{M}$ is a diffeomorphism $k:A\times B\rightarrow \hat{M}$ for some manifolds $A$ and $B$. For further details see John M. Lee's \textit{Introduction to Smooth Manifolds} \cite{lee2003smooth}. If a Seifert-fibered manifold $\hat{M}$ has a product structure $k:S^{1}\times F\rightarrow \hat{M}$ for some surface with boundary $F$ and $k(S^{1}\times\{x\})$ are the fibers of $\hat{M}$ for each $x\in F$, then we say that $k:S^{1}\times F\rightarrow \hat{M}$ is a \textit{fibering product structure} of $\hat{M}$. 

We note here that a fibering product structure on $\hat{M}$ is equivalent to the existence of a foliation of $\hat{M}$ by both circles and by surfaces diffeomorphic to $F$ so that any circle intersects each foliated surface exactly once.

Given that the first homology group (equivalently the first fundamental group) of a torus is $\mathbb{Z}\times\mathbb{Z}$ generated by two elements represented by any two nontrivial loops that cross at a single point, we can use the \textit{meridian-longitude framing} from a product structure as representatives of two generators. If we have a diffeomorphism $f:T_{1}\rightarrow T_{2}$ and product structures $k_{i}:S^{1}\times S^{1}\rightarrow T_{i}$, then we can express the induced map between the first homology groups $H_{1}(T_{1})$ and $H_{1}(T_{2})$ by a matrix that uses bases for $H_{1}(T_{i})$ derived from the meridian-longitude framings that arise from $k_{i}:S^{1}\times S^{1}\rightarrow T_{i}$. We denote this matrix as $\left[\begin{array}{cc}
a_{11} & a_{12}\\
a_{21} & a_{22}
\end{array}\right]_{k_{2}}^{k_{1}}:H_{1}(T_{1}) \rightarrow H_{1}(T_{2})$.

We say that a $G$-action $\varphi:G\rightarrow Diff(A\times B)$ is a \textit{product action} if for each $g\in G$, the diffeomorphism $\varphi(g):A\times B\rightarrow A\times B$ can be expressed as $(\varphi_{1}(g),\varphi_{2}(g))$ where $\varphi_{1}(g):A\rightarrow A$ and $\varphi_{2}(g):B\rightarrow B$. Here $\varphi_{1}:G\rightarrow Diff(A)$ and $\varphi_{2}:G\rightarrow Diff(B)$ are not necessarily injections. 

Given an action $\varphi:G\rightarrow Diff(M)$ and a product structure $k:A\times B\rightarrow M$, we say that $\varphi$ \textit{leaves the product structure $k:A\times B\rightarrow M$ invariant} if $\psi(g)=k^{-1}\circ\varphi(g)\circ k$ defines a product action $\psi:G\rightarrow Diff(A\times B)$.

If we have a manifold $\hat{M}$ with torus boundary components and each of those boundary tori $T_{i}$ have a product structure $k_{i}:S^{1}\times S^{1}\rightarrow T_{i}$, then we say a $G$-action $\varphi:G\rightarrow Diff(\hat{M})$ \textit{respects the product structures} on the boundary tori if $k_{j}^{-1}\circ\varphi(g)\circ k_{i}:S^{1}\times S^{1}\rightarrow S^{1}\times S^{1}$ can be expressed as $(\varphi_{1}(g),\varphi_{2}(g))$ where $\varphi_{1}:G\rightarrow Diff(S^{1})$ and $\varphi_{2}:G\rightarrow Diff(S^{1})$. These again are not necessarily injections.

Suppose that we now have a fibering product structure $k:S^{1}\times F\rightarrow M$. We then say that each boundary torus is \textit{positively oriented} if the fibers are given an arbitrary orientation and then each boundary component of $k(\{u\}\times F)$ is oriented by taking the normal vector to the surface according the orientation of the fibers.

We will throughout treat $S^{1}$ as the unit circle within $\mathbb{C}$ and by extension the unit disc will be $D=\{ru|0\leq r \leq 1, u \in \mathbb{C}, ||u||=1\}$; the torus will be $T=S^{1} \times S^{1}$; and the solid torus will be $V=S^{1}\times D$.

\section{Dehn Fillings and Seifert Manifolds}

We first establish some background work on Dehn fillings and Seifert manifolds by showing how a manifold $M$ can be constructed by filling the boundary tori of some product manifold $\hat{M}=S^{1} \times F$ with solid fibered tori.

This section broadly follows the construction from the work of Mark Jankins and Walter Neumann in \textit{Lectures on Seifert manifolds} \cite{jankins1983lectures}. We will use the following notation for a compact, closed, and orientable Seifert manifold $M$ with orientable base space:$$(g,o_{1}|(q_{1},p_{1}),\ldots,(q_{n},p_{n})), q_{i}>0$$

This notation implies that $M$ is a manifold that can be decomposed into a manifold $\hat{M}\cong S^{1}\times F$ that is trivially fibered with boundary $\partial\hat{M}=T_{1}\cup\ldots\cup T_{n}$, and $X=V_{1}\cup\ldots\cup V_{n}$, a disjoint collection of fibered solid tori (the notation specifies the fibration). Here $F$ is a compact, connected, orientable genus $g$ surface with $n$ boundary components. $M$ is reobtained by a gluing map $d:\partial X\rightarrow\partial\hat{M}$. This is defined as follows:

Take a given fibering product structure $k_{\hat{M}}:S^{1}\times F\rightarrow\hat{M}$ on $\hat{M}$, and some particular product structure $k_{X}:S^{1}\times(D_{1}\cup\ldots\cup D_{n})\rightarrow X$ where each $D_i$ is a disk. Then define product structures $k_{\partial V_{i}}:S^{1}\times S^{1}\rightarrow\partial V_{i}$ and $k_{T_{i}}:S^{1}\times S^{1}\rightarrow T_{i}$ by parameterizing each component of $\partial F$ and $\partial D_{i}$ with a positive orientation by some diffeomorphisms $\rho_{i}:S^{1}\rightarrow(\partial F)_{i}$ and $\sigma_{i}:S^{1}\rightarrow\partial D_{i}$, and then taking $k_{\partial V_{i}}(u,v)=k_{X}(u,\sigma_{i}(v))$ and $k_{T_{i}}(u,v)=k_{\hat{M}}(u,\sigma_{i}(v))$.

$d:\partial X\rightarrow\partial\hat{M}$ is then a diffeomorphism such that $d(\partial V_{i})=T_{i}$ and $$(k_{T_{i}}^{-1}\circ d|_{\partial Vi}\circ k_{\partial V_{i}})(u,v)=(u^{x_{i}}v^{p_{i}},u^{y_{i}}v^{q_{i}})$$ where $x_{i}q_{i}-y_{i}p_{i}=-1$ and $|y_{i}|<q_{i}$. 

This condition requires that $(q_{i},p_{i})$ are coprime. 

We note therefore that the induced fibration on each solid torus $V_{i}$, is a $(-q_{i},y_{i})$ fibration (according to $k_{\partial V_{i}}$). Hence $(q_{i},p_{i})$ refers to a regular fiber if $q_{i}=\pm 1$ and a critical fiber otherwise. Also note that again by compactness there can only be a finite number of critical fibers.

We now quote Theorem 1.1. from Walter Neumann and Frank Raymond's paper \textit{Seifert manifolds, plumbing, $\mu$-invariant and orientation reversing maps} \cite{neumann1978seifert} regarding Seifert invariants:

\begin{thm}
Let $M$ and $M'$ be two orientable Seifert manifolds with associated Seifert invariants $(g,o_{1}|(\alpha_{1},\beta_{1}),\ldots,(\alpha_{s},\beta_{s}))$ and $(g,o_{1}|(\alpha_{1}',\beta_{1}'),\ldots,(\alpha_{t}',\beta_{t}'))$ respectively. Then $M$ and $M'$ are orientation-preservingly diffeomorphic by a fiber-preserving diffeomorphism if and only if, after reindexing the Seifert pairs if necessary, there exists an n such that:
\begin{enumerate}
 \item $\alpha_{i}=\alpha_{i}'$ for $i=1,\ldots,n$ and $\alpha_{i}=\alpha_{j}'=1$ for $i,j>n$ 
\item $\beta_{i}\equiv\beta_{i}'\textrm{ (mod }\alpha_{i})$ for $i=1,\ldots,n$
\item $\sum\limits_{i=1}^{s}\frac{\beta_{i}}{\alpha_{i}}=\sum\limits_{i=1}^{t}\frac{\beta_{i}'}{\alpha_{i}'}$
\end{enumerate}
\end{thm}

The consequence of this theorem is that we can perform the following "moves" on the Seifert invariants:
\begin{enumerate}
    \item Permute the indices
\item Add or delete a Seifert pair (1,0)
\item Replace $(\alpha_{1},\beta_{1}), (\alpha_{2},\beta_{2})$ by $(\alpha_{1},\beta_{1}+m\alpha_{1}), (\alpha_{2},\beta_{2}-m\alpha_{2})$ for some integer $m$.
\end{enumerate}

From this we yield the Corollary:

\begin{cor} Let $M$ and $M'$ be two orientable Seifert manifolds with associated Seifert invariants $(g,o_{1}|(\alpha_{1},\beta_{1}),\ldots,(\alpha_{s},\beta_{s}))$ and $(g,o_{1}|(\alpha_{1},\beta_{1}+m_{1}\alpha_{1}),\ldots,(\alpha_{s},\beta_{s}+m_{s}\alpha_{s}))$ respectively. Then $M$ and $M'$ are orientation-preservingly diffeomorphic by a fiber-preserving diffeomorphism if and only if $$\sum_{i=1}^{s}m_{i}=0$$

\end{cor}

\begin{proof} By Theorem 2.1, we need only consider the third condition. The first two conditions hold trivially. So, the two manifolds are diffeomorphic if and only if: $$\sum\limits_{i=1}^{s}\frac{\beta_{i}}{\alpha_{i}}=\sum\limits_{i=1}^{s}\frac{\beta_{i}+m_{i}\alpha_{i}}{\alpha_{i}}=\sum\limits_{i=1}^{s}\frac{\beta_{i}}{\alpha_{i}}+\sum\limits_{i=1}^{s}m_{i}$$

Hence, if and only if $$\sum_{i=1}^{s}m_{i}=0$$

\end{proof}

We can now define normalized Seifert invariants so that any orientable Seifert manifold over an orientable base space can be expressed as: $$(g,o_{1}|(q_{1},p_{1}),\ldots,(q_{n},p_{n}),(1,b))$$

Where $0<p_{i}<q_{i}$ and $b$ is some integer called the \textit{obstruction class}. 

The constant: $$e=-(b+\sum_{i=1}^{n}\frac{p_{i}}{q_{i}})$$ is known as the \textit{Euler class of the Seifert bundle} and is zero if and only if the Seifert bundle is covered by the trivial bundle. Alternatively, it is zero if the manifold $M$ has the geometry of either $S^{2}\times\mathbb{R},H^{2}\times\mathbb{R}$, or $E^{3}$. For more details, refer to Peter Scott's paper \textit{The geometries of 3-manifolds} \cite{scott1983geometries}.

\section{Construction of a Finite, Fiber- and Orientation-Preserving Action}

We now present a construction for a finite, orientation and fiber-preserving action on a Seifert manifold $M=(g,o_{1}|(q_{1},p_{1}),\ldots,(q_{n},p_{n}))$. Here the Seifert invariants are not necessarily normalized.

According to Section 2, we can decompose $M$ into $\hat{M}$ and $X$ where $\hat{M}\cong S^{1}\times F$ is trivially fibered and $X$ is a disjoint union of $n$ solid tori. We then have a gluing map $d:\partial X\rightarrow\partial\hat{M}$, so that for a fibering product structure $k_{\hat{M}}:S^{1}\times F\rightarrow\hat{M}$, there is some $k_{X}:S^{1}\times(D_{1}\cup\ldots\cup D_{n})\rightarrow X$ and restricted positively oriented product structures $k_{\partial V_{i}}:S^{1}\times S^{1}\rightarrow\partial V_{i}$ and $k_{T_{i}}:S^{1}\times S^{1}\rightarrow T_{i}$ such that $(k_{T_{i}}^{-1}\circ d|_{\partial V_{i}}\circ k_{\partial V_{i}})(u,v)=(u^{x_{i}}v^{p_{i}},u^{y_{i}}v^{q_{i}})$.

\subsection{Constructing a Finite, Fiber-Preserving Action on $\hat{M}$}

We pick a finite, fiber-preserving group action on $\hat{M}$ by first choosing some (not-necessarily effective) group action $\varphi_{1}:G\rightarrow Diff(S^{1})$. This will necessarily be of the form:$$\varphi_{1}(g)(u)=\theta_{1}(g)u^{\alpha(g)}$$

Here $\theta_{1}:G\rightarrow S^{1}$ and $\alpha:G\rightarrow\{-1,1\}$. The precise nature of these maps is shown in Section 3.5.

We then choose a (not-necessarily effective) group action $\varphi_{2}:G\rightarrow Diff(F)$ such that if we parameterize each component of $\partial F$ in the same way as in Section 2 and then express $\partial F=\{(v,i)|v\in S^{1},i\in\{1,\ldots,n\}\}$, we can write: $$\varphi_{2}(g)|_{\partial F}(v,i)=(\theta_{2}(i,g)v^{\alpha(g)},\beta(g)(i))$$

Here $\theta_{2}:\{1,\ldots,n\}\times G\rightarrow S^{1}$, and $\beta:G\rightarrow perm(\{1,\ldots,n\})$ are such that $\beta(g)(i)=j$ only if $(q_{i},p_{i})=(q_{j},p_{j})$.

Then we define our group action $\varphi:G\rightarrow Diff(\hat{M})$ by: $$(k_{\hat{M}}^{-1}\circ\varphi(g)\circ k_{\hat{M}})(u,x)=(\varphi_{1}(g)(u),\varphi_{2}(g)(x))$$

So now we can fully express $\varphi:G\rightarrow Diff(\hat{M})$ on the boundary of $\hat{M}$ by: $$(k_{T_{\beta(g)(i)}}^{-1}\circ\varphi(g)\circ k_{T_{i}})(u,v)=(\theta_{1}(g)u^{\alpha(g)},\theta_{2}(i,g)v^{\alpha(g)})$$

We note here that (according to the set framing of each boundary torus), each element $g\in G$ acts on a boundary tori $T_{i}$ by mapping it to $T_{\beta(g)(i)}$ with:

\begin{itemize}
\item a rotation by $\theta_{1}(g)$ in the longitudinal direction.

\item a rotation by $\theta_{2}(i,g)$ in the meridianal direction.

\item a reflection in the meridian and longitude if $\alpha(g)=-1$.
\end{itemize}

\subsection{Inducing a Finite, Fiber-Preserving Action on $\partial X$}

We can now induce an action on $\partial X$ by: $$\psi:G\rightarrow Diff(\partial X)$$ $$\psi(g)=d^{-1}\circ\varphi(g)|_{\partial\hat{M}}\circ d$$

This we can fully express (after simplification) as: $$(k_{\partial V_{\beta(g)(i)}}^{-1}\circ\psi(g)\circ k_{\partial V_{i}})(u,v)
	=(\theta_{1}(g)^{-q_{i}}\theta_{2}(i,g)^{p_{i}}u^{\alpha(g)},\theta_{1}(g)^{y_{i}}\theta_{2}(i,g)^{-x_{i}}v^{\alpha(g)})$$
    
Therefore - according to the set framing of each boundary torus) - each element $g\in G$ acts on a $\partial V_{i}$ by mapping it to $\partial V_{\beta(g)(i)}$ with:

\begin{itemize}
\item a rotation by $\theta_{1}(g)^{-q_{i}}\theta_{2}(i,g)^{p_{i}}$ in the longitudinal direction.

\item a rotation by $\theta_{1}(g)^{y_{i}}\theta_{2}(i,g)^{-x_{i}}$ in the meridianal direction.

\item a reflection in the meridian and longitude if $\alpha(g)=-1$.
\end{itemize}

 Alternatively, we could view this action by each element $g\in G$ mapping $\partial V_{i}$ to $\partial V_{\beta(g)(i)}$ with:

\begin{itemize}
\item a rotation by $\theta_{1}(g)$ along $(-q_{j},y_{j})$ curves (along the fibers).

\item a rotation by $\theta_{2}(i,g)$ along $(p_{j},-x_{j})$ curves.

\item a reflection in the meridian and longitude if $\alpha(g)=-1$.
\end{itemize}

\subsection{Extending the Induced Action to $X$.}

We have that:

\[k_{X}^{-1}(X)=\{(u,v,i)|u\in S^{1},v\in D,i\in\{1,\ldots,n\}\}\]

Where $D$ is the unit disc. Hence the action $\psi:G\rightarrow Diff(X)$ straightforwardly extends by coning inwards. This works as the product structure on $X$ is such that the fibration is normalized. Hence, the extended action is fiber-preserving.

\subsection{The Final Action}

So now we have defined finite, fiber- and orientation-preserving actions on $\hat{M}$ and $X$ such that they agree under the gluing map $d:\partial X\rightarrow\partial\hat{M}$. This completes the construction.

We now formally make the definition that we refer to any action $\varphi:G\rightarrow Diff^{fp}_{+}(M)$ that can be constructed as above as an \textit{extended product action}.

We close this subsection with a brief, notable remark:

\begin{rem}

Note that in these examples $\varphi_{1}:G\rightarrow Diff(S^{1})$ and $\varphi_{2}:G\rightarrow Diff(F)$ are not injections in all cases and so not necessarily effective actions.

\end{rem}

\subsection{Conditions for $\varphi_{1}:G\rightarrow Diff(S^{1})$ and $\varphi_{2}:G\rightarrow Diff(F)$}

We here establish some necessary and sufficient conditions in the construction of $\varphi_{1}:G\rightarrow Diff(S^{1})$ and $\varphi_{2}:G\rightarrow Diff(F)$.

\begin{prop}
 
The following are necessary and sufficient conditions on $\theta_{1}:G\rightarrow S^{1}$ and $\alpha:G\rightarrow\{-1,1\}$ for $\varphi_{1}:G\rightarrow Diff(S^{1})$ to be a homomorphism:
\begin{enumerate}
    \item $\alpha:G\rightarrow\{-1,1\}$ is a homomorphism.
     \item $\theta_{1}(g_{1}g_{2})=\theta_{1}(g_{1})\theta_{1}(g_{2})^{\alpha(g_{1})}$ 
\end{enumerate}
\end{prop}

\begin{proof}  We calculate $\varphi_{1}(g_{1}g_{2})(u)=\theta_{1}(g_{1}g_{2})u^{\alpha(g_{1}g_{2})}$ and:
$$\varphi_{1}(g_{1})\circ\varphi_{1}(g_{2})(u)=\theta_{1}(g_{1})(\theta_{1}(g_{2})u^{\alpha(g_{2})})^{\alpha(g_{1})}=\theta_{1}(g_{1})\theta_{1}(g_{2})^{\alpha(g_{1})}u^{\alpha(g_{2})\alpha(g_{1})}$$
These are equal for all values of $u$. Hence for $u=1$ we have that $\theta_{1}(g_{1}g_{2})=\theta_{1}(g_{1})\theta_{1}(g_{2})^{\alpha(g_{1})}$.

This establishes part ii) and then implies that $u^{\alpha(g_{1}g_{2})}=u^{\alpha(g_{1})\alpha(g_{2})}$ which establishes part i).
\end{proof}

\begin{prop}
 The following are necessary conditions on $\theta_{2}:\{1,\ldots,n\}\times G\rightarrow S^{1}$, $\alpha:G\rightarrow\{-1,1\}$, and $\beta:G\rightarrow perm(\{1,\ldots,n\})$ if $\varphi_{2}:G\rightarrow Diff(F)$ is a homomorphism:
\begin{enumerate}
    \item $\alpha:G\rightarrow\{-1,1\}$ is a homomorphism.
    \item $\beta:G\rightarrow perm(\{1,\ldots,n\})$ is a homomorphism.
    \item $\theta_{2}(i,g_{1}g_{2})=\theta_{2}(\beta(g_{2})(i),g_{1})\theta_{2}(i,g_{2})^{\alpha(g_{1})}$
\end{enumerate}
\end{prop}

\begin{proof}
 We first calculate $\varphi_{2}(g_{1}g_{2})(v,i)=(\theta_{2}(i,g_{1}g_{2})v^{\alpha(g_{1}g_{2})},\beta(g_{1}g_{2})(i))$. Then calculate:
 \begin{align*}
 \varphi_{2}(g_{1})\circ\varphi_{2}(g_{2})(v,i)	&=\varphi_{2}(g_{1})(\theta_{2}(i,g_{2})v^{\alpha(g_{2})},\beta(g_{2})(i))\\
	&=(\theta_{2}(\beta(g_{2})(i),g_{1})(\theta_{2}(i,g_{2})v^{\alpha(g_{2})})^{\alpha(g_{1})},\beta(g_{2})\circ\beta(g_{1})(i))\\
	&=(\theta_{2}(\beta(g_{2})(i),g_{1})\theta_{2}(i,g_{2})^{\alpha(g_{1})}v^{\alpha(g_{1})\alpha(g_{2})},\beta(g_{2})\circ\beta(g_{1})(i))\\
\end{align*} 

These are again equal for all values of $v$ and $i$. We immediately have that $\beta(g_{1}g_{2})=\beta(g_{1})\circ\beta(g_{2})$ and part ii) follows. 

Now, for $v=1$ we have that $\theta_{2}(i,g_{1}g_{2})=\theta_{2}(\beta(g_{2})(i),g_{1})\theta_{2}(i,g_{2})^{\alpha(g_{1})}$.

This establishes part iii) and leaves $v^{\alpha(g_{1}g_{2})}=v^{\alpha(g_{1})\alpha(g_{2})}$ which establishes part i).
\end{proof}

\section{Actions on $\hat{M}$}

In order to find out to what extent finite, fiber- and orientation-preserving actions are extended product actions, we first need to establish a result regarding actions on $\hat{M}$. In this section we always take $F$ to be an orientable surface with boundary and $\hat{M}$ to be the fibered manifold that has boundary made up of tori described earlier.

The main result we prove in this section is an adaptation of Theorem 2.3 in \cite{Meeks1986}. It will state that if $\hat{M}$ has a product structure, then there is another product structure on $\hat{M}$ that remains invariant under the group action provided the restricted product structures on each boundary component are respected by the action. Moreover, the two product structures foliate the boundary tori identically.

We first state some preliminary results. 

\begin{lem}
 Let $\varphi:G\rightarrow Diff(F)$ be a finite group action with $F$ not a disc. Then $F$ contains a $\varphi$-equivariant essential simple arc. 
 \end{lem}

\begin{proof}
 $F/\varphi$ is a $2$-orbifold. We can then pick an essential simple arc in the underlying space of $F/\varphi$ that doesn't intersect the cone points and then lift this to a $\varphi$-equivariant essential simple arc in $F$.
 \end{proof}

\begin{lem}
 Let $\psi:G\rightarrow Diff_{+}^{fp}(T)$ be a finite group action on a Seifert-fibered torus. Suppose that there exists a fibering product structure $k:S^{1}\times S^{1}\rightarrow T$. Then $\psi:G\rightarrow Diff_{+}^{fp}(T)$ is equivalent to a fiber-preserving group action that leaves the product structure $k:S^{1}\times S^{1}\rightarrow T$ invariant. Moreover, the conjugating map is fiber-preserving and isotopic to the identity.
 \end{lem}

\begin{proof}
 First note that necessarily, $\psi(g)_{*}=\pm\left[\begin{array}{cc}
1 & 0\\
0 & 1
\end{array}\right]_{k}^{k}$. This follows from the fact that $\pm\left[\begin{array}{cc}
1 & c\\
0 & 1
\end{array}\right]$ has finite order only if $c=0$.

We then note that by \cite{thurstongeometry}, the only possible quotient types are a torus or $S^{2}(2,2,2,2)$. By John Kalliongis and Andy Miller in \textit{The symmetries of genus one handlebodies} \cite{kalliongis1991symmetries} these refer respectively to actions of groups $\mathbb{Z}_{m}\times\mathbb{Z}_{n}$ and $Dih(\mathbb{Z}_{m}\times\mathbb{Z}_{n})$ where $\mathbb{Z}_{m}\times\mathbb{Z}_{n}$ acts by preserving the orientation of the fibers and the dihedral $\mathbb{Z}_{2}$ subgroup of $Dih(\mathbb{Z}_{m}\times\mathbb{Z}_{n})$ acts by reversing the orientation of the fibers.

We first consider the torus case. This will receive an induced fibration from $T$. We can then pick a fibering product structure on $T/\psi$ . This product structure can be lifted to an invariant fibering product structure $k':S^{1}\times S^{1}\rightarrow T$. According to this product structure, the group acts as rotations along the fibers or along loops $k'(\{u\}\times S^{1})$. As such, it preserves any fibration up to isotopy. So we can assume that $k':S^{1}\times S^{1}\rightarrow T$ is in fact isotopic to the original product structure $k:S^{1}\times S^{1}\rightarrow T$.

We then let $f=k'\circ k^{-1}$. So that $k^{-1}\circ f^{-1}\circ\psi(g)\circ f\circ k=k'^{-1}\circ\psi(g)\circ k'$

This is a product. It also follows that $f$ is fiber-preserving and isotopic to the identity.

If the action has quotient of $S^{2}(2,2,2,2)$, then we note that as the fiber orientation-preserving subgroup $\mathbb{Z}_{m}\times\mathbb{Z}_{n}$ is a normal subgroup, we can consider the induced $\mathbb{Z}_{2}$-action on the quotient of the $\mathbb{Z}_{m}\times\mathbb{Z}_{n}$-action. This is necessarily a “spin” action by \cite{kalliongis1991symmetries} and we can pick a fibering product structure on $T/(\mathbb{Z}_{m}\times\mathbb{Z}_{n})$ as above but that further remains invariant under the “spin” action.
\end{proof}

\begin{lem}
 Let $k:S^{1}\times F\rightarrow\hat{M}$ and $k':S^{1}\times F\rightarrow\hat{M}$ be fibering product structures so that they foliate the boundary tori identically. Then $k(\{1\}\times F)$ is freely isotopic to $k'(\{1\}\times F)$.
 \end{lem}

\begin{proof}
 Consider, $k'^{-1}\circ k:S^{1}\times F\rightarrow S^{1}\times F$. Necessarily, this can be expressed in the form $(k'^{-1}\circ k)(u,x)=(k_{1}(u,x),k_{2}(x))$.

So now by composing with the diffeomorphism $l:S^{1}\times F\rightarrow S^{1}\times F$ given by $l(u,x)=(u,k_{2}^{-1}(x))$, we have that $(k'^{-1}\circ k\circ l)(u,x)=(k_{1}(u,x),x)$.

Consider $(k\circ l)(S^{1}\times\{x\})$ and $(k')(S^{1}\times\{x\})$. These are the same fiber. Hence $(k\circ l)(\{1\}\times F)$ and $(k')(\{1\}\times F)$ are freely isotopic by isotoping along the fibers. 
\end{proof}

The final required result is the equivariant Dehn's Lemma. We state it here in the form used by Allan Edmonds in his paper \textit{A topological proof of the equivariant Dehn lemma} \cite{edmonds1986topological}.

\begin{lem}
Let $\varphi:G\rightarrow Diff(\hat{M})$ be a finite group action. Let $\gamma\subset\partial\hat{M}$ be a simple closed curve such that $\gamma$ is:
\begin{enumerate} 
\item null-homotopic in $\hat{M}$.
\item $\varphi$-equivariant.
\item transverse to the exceptional set of $\varphi$. 
\end{enumerate}
Then there exists an embedded disc D such that: 
\begin{enumerate}
\item $\gamma=\partial D$
\item $D$ is $\varphi$-equivariant.
\item $D$ is transverse to the exceptional set of $\varphi$.
\end{enumerate}
\end{lem}

The proof of the theorem then follows that of \cite{Meeks1986} in an adapted and expanded form.

\begin{thm}
 Let $k:S^{1}\times F\rightarrow\hat{M}$ be a fibering product structure such that the finite group action $\varphi:G\rightarrow Diff_{+}^{fp}(\hat{M})$ respects the restricted product structures on each boundary torus. Then there exists an isotopic fibering product structure $k':S^{1}\times F\rightarrow\hat{M}$ such that the group action $\psi:G\rightarrow Diff(S^{1}\times F)$ given by $\psi(g)=k'^{-1}\circ\varphi(g)\circ k'$ for each $g\in G$ is a product action and foliates the boundary identically to $k$.
 \end{thm}

\begin{proof}
 We proceed by induction on the Euler characteristic of $F$.

\noindent \underline{Initial Case: $\chi(F)=1$}

We therefore have $\hat{M}$ as a trivially fibered solid torus with $k:S^{1}\times F\rightarrow\hat{M}$, a fibering product structure. By the product structure on the boundary, we have a foliation by meridianal circles that each bound a disc and the usual longitudinal Seifert fibration by circles. So any of the meridianal circles are necessarily $\varphi$-equivariant. Then taking such a circle, we apply the equivariant Dehn's Lemma (Lemma 4.5) to yield a $\varphi$-equivariant disc $D$ whose boundary agrees with the product structure on the boundary of the solid torus. We now decompose along $Orb(D)=\{D_{1},\ldots,D_{s}\}$ to yield a collection $B_{1},\ldots,B_{s}$ of balls, each which are homeomorphic to $I\times D$ and fibered by arcs.

So starting with $B_{1}$  we have the action $\varphi_{1}:Stab(B_{1})\rightarrow Diff(B_{1})$ given by $\varphi_{1}(g)=\varphi(g)|_{B_{1}}$. 

Note that the quotient orbifold $B_{1}/\varphi_{1}$ necessarily has boundary either $S^{2}(n,n)$ or $S^{2}(2,2,n)$. This follows from John  Kalliongis and Ryo Ohashi in their paper \textit{Finite actions on the $2$-sphere, the projective plane and $I$-bundles over the projective plane} \cite{kalliongis2018}, where they show that these are the only orientable quotients of $S^{2}$ where the action fixes one point or exchanges two points (corresponding to the two discs $D_{1},D_{2}$).

We here use the proof of the Smith conjecture (see ball orbifolds in Francis Bonahon's \textit{Geometric structures on 3-manifolds} \cite{sher2001handbook}) to see that $B_{1}/\varphi_{1}$ has the following possible forms with induced (orbifold) foliations on part of the boundary shown by Figure 1.

\begin{figure}[ht]
\centering
\includegraphics[height=3cm]{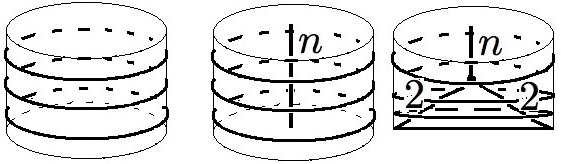}
\caption{Possible quotients with induced orbifold foliations on part of the boundary}
\end{figure}

On the part of the boundary that lifts into $\partial\hat{M}$, the first two are foliated simply by circles, and the third is foliated by circles and one $1$-orbifold with cone points of order $2$ on either end.

This first can then clearly be foliated by discs that agree with the foliation by circles on the boundary. The second can be foliated by discs with a cone point of order $n$ with the discs agreeing with the foliation by circles on the boundary. 

The third can be foliated by discs with cone points order $n$ - with the discs having boundaries given by the circles - and a $2$-orbifold of the form shown in Figure 2. This has Thurston data set given by $(;n)$.

\begin{figure}[ht]
\centering
\includegraphics[height=1.5cm]{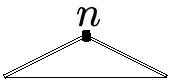}
\caption{An element of the orbifold foliation of the third possible $B_{1}/\varphi_{1}$}
\end{figure}

Each of these can be taken to intersect each induced orbifold $I$-fiber once and will lift to an invariant foliation of $B_{1}$ by discs that each intersect each $I$-fiber once.

We therefore have a product structure $k_{1}:I\times F\rightarrow B_{1}$ that remains invariant under the action $\varphi_{1}:Stab(B_{1})\rightarrow Diff(B_{1})$. Furthermore, its' foliation (by arcs and circles) on the part of its boundary that intersects with the boundary of $\hat{M}$ is equal to the restricted foliation from $k:S^{1}\times F\rightarrow\hat{M}$.

We now translate to the remaining $B_{i}$. For each $B_{i}$, there is some $g_{i}\in G$ such that $\varphi(g_{i})(B_{1})=B_{i}$ and we can then define product structures $k_{i}:I\times F\rightarrow B_{i}$ by $k_{i}=\varphi(g_{i})\circ k_{1}$. 

Note that as each $\varphi(g_{i})$ leaves the original product structure $k:S^{1}\times F\rightarrow\hat{M}$ invariant on the boundary of $\hat{M}$ then each $k_{i}:I\times F\rightarrow B_{i}$ foliates $B_{i}$ (by arcs and circles) on the part of its' boundary that intersects with the boundary of $\hat{M}$ the same way as the restricted foliation from $k:S^{1}\times F\rightarrow\hat{M}$.

Then for any $g\in G$ such that $\varphi(g)(B_{i})=B_{j}$ we have $g=g_{j}hg_{i}^{-1}$ for some $h\in Stab(B_{1})$ and can calculate $k_{j}^{-1}\circ\varphi(g)\circ k_{i}=k_{1}^{-1}\circ\varphi(h)\circ k_{1}$. This is a product by above.

So now we have a collection of product structures on each $B_{1},\ldots,B_{s}$ that remain invariant under the action. We view these now as invariant foliations by arcs and discs. By construction, we yield invariant foliations of $\hat{M}$ by circles and discs. This is possible as each of the invariant foliations of $B_{i}$ are equal to the restricted foliation from $k:S^{1}\times F\rightarrow\hat{M}$ on the part of its' boundary that intersects with the boundary of $\hat{M}$.

These invariant foliations give our required $k':S^{1}\times F\rightarrow\hat{M}$.

\noindent \underline{Inductive Step:}

We now fix an integer $c<1$ and suppose the result holds for $\chi(F)>c$. We proceed to prove the case where $\chi(F)=c$ by induction.

Our strategy is to break $\hat{M}$ into pieces each of which fibers over a surface with Euler characteristic greater than $c$. We can then apply the inductive hypothesis before reassembling $\hat{M}$ and deriving the result for $\chi(F)=c$.

We induce the action $\varphi_{F}:G_{F}\rightarrow Diff(F)$ on the base space of the fibration and then apply Lemma 4.1 to yield a $\varphi_{F}$-equivariant essential simple arc in $F$. We call this arc $\lambda$ and define $A_{1}$ to be the annulus made up of fibers that project to $\lambda$. As $\varphi:G\rightarrow Diff(\hat{M})$ is fiber-preserving, this is necessarily $\varphi$-equivariant.

Cutting along the collection of annuli $Orb(A_{1})$ will yield a disjoint collection $\{\hat{M}_{1},\ldots,\hat{M}_{n}\}$ of manifolds with boundary which fiber over surfaces $\{F_{1},\ldots,F_{n}\}$. Necessarily, each of these have greater Euler number than $F$. 

Now pick $\hat{M}_{1}$ and pick any boundary torus $T$ of $\hat{M}_{1}$ that contains $A_{1}$. This consists of annuli that were originally contained in a boundary torus of $\hat{M}$ before being cut open - we refer to these as $A'_{1},\ldots,A'_{m}$ - or some annuli in the collection $Orb(A_{1})$ - we refer to these as $A_{1},\ldots,A_{m}$. Note that there must be an equal number of each type of annulus. Each of $A'_{1},\ldots,A'_{m}$ inherit product structures $k_{A'_{i}}:S^{1}\times I\rightarrow A'_{i}$ that are respected under the restricted action of $Stab(T)$.

Now consider $T/Stab(T)$. This will necessarily be either another torus consisting of two glued annuli - one referring to the projection of $A_{1}$ and the other referring to the projection of $A'_{1}$ - or an $S^{2}(2,2,2,2)$ consisting of two glued together $D(2,2)$ - again, one referring to the projection of $A_{1}$ and the other referring to the projection of $A'_{1}$. This follows from \cite{thurstongeometry}.

\noindent \textbf{Case 1: $T/Stab(T)$ is a torus.}

The annulus covered by $A'_{1}$ has an induced Seifert fibration and foliation by arcs. The annulus covered by $A_{1}$ has an induced Seifert fibration and can by foliated by arcs so that $T/Stab(T)$ is foliated by circles that cross each fiber once. 

\noindent \textbf{Case 2: $T/Stab(T)$ is $S^{2}(2,2,2,2)$}

The $D(2,2)$ covered by $A'_{1}$ has an induced orbifold Seifert fibration and orbifold foliation as shown below in Figure 3. The $D(2,2)$ covered by $A_{1}$ has an induced orbifold Seifert fibration and can be orbifold foliated so that $T/Stab(T)$ is orbifold foliated so that each leaf of the foliation crosses each fiber once. 

\begin{figure}[ht]
\centering
\includegraphics[height=3.5cm]{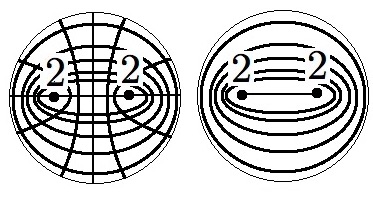}
\caption{The two $D(2,2)$ covered by $A'_{1}$ and $A_{1}$}
\end{figure}

Moreover these orbifold foliations can be chosen so that they lift to give $T$ a foliation that is invariant under $Stab(T)$; agrees with the foliation by arcs given by $k_{A'_{i}}:S^{1}\times I\rightarrow A'_{i}$; and is isotopic to the induced foliation of $T$ from the original $k:S^{1}\times F\rightarrow\hat{M}$. This follows from Lemma 4.2. 

This then defines a product structure $k_{T}:S^{1}\times S^{1}\rightarrow T$ invariant under the action of $Stab(T)$ which restricts to a product structure $k_{A_{1}}:S^{1}\times I\rightarrow A_{1}$ invariant under $Stab(A_{1})$.

We now translate to each $T_{i}\in Orb_{Stab(\hat{M}_{1})}(T)$ by taking some $g_{i}\in G$ such that $\varphi(g_{i})(T)=T_{i}$. We then define product structures $k_{T_{i}}:S^{1}\times S^{1}\rightarrow T_{i}$ by $k_{T_{i}}=\varphi(g_{i})\circ k_{T_{1}}$.

For any $g\in G$ with $\varphi(g)(T_{i})=T_{j}$ for some $i,j$, we have that $g=g_{j}g'g_{i}^{-1}$ for some $g'\in Stab(T_{1})$. So then $k_{T_{j}}^{-1}\circ\varphi(g)\circ k_{T_{i}}=k_{T_{1}}^{-1}\circ\varphi(g')\circ k_{T_{1}}$. Hence it is a product and the product structures on each of the tori $T_{i}$ are respected under $Stab(\hat{M}_{1})$. 

We do this for each orbit of boundary components of $\hat{M}_{1}$ to yield product structures on each boundary tori that are respected under $Stab(\hat{M}_{1})$ and that agree with the inherited product structure from the original boundary of $\hat{M}$. 

We then translate these product structures to the boundaries of each $\hat{M}_{i}$.

We can now begin to reconstruct $\hat{M}$ and we can assume that we have respected product structures on each of the connected components of the union of $\partial\hat{M}$ and $Orb(A_{1})$. Pick the first connected component $C$ that yielded $T$ when we cut as shown in Figure 4.

\begin{figure}[ht]
\centering
\includegraphics[height=3.5cm]{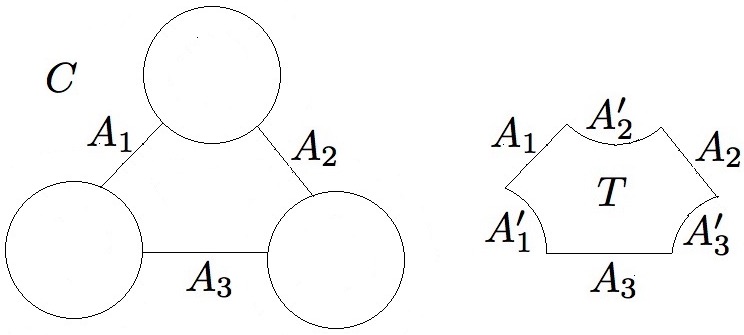}
\caption{A connected component $C$ of the union of $\partial\hat{M}$ and $Orb(A_{1})$}
\end{figure}

The product structure on this connected component is necessarily isotopic to the original product structure by construction. Suppose that the product structure on some other connected component $C'$ was defined by translating by $\varphi(g)$. We now note that $k:S^{1}\times F\rightarrow\hat{M}$ and $\varphi(g)\circ k:S^{1}\times F\rightarrow\hat{M}$ satisfy the requirements of Lemma 4.3. Hence applying the lemma, we yield that the restricted product structure on $C'$ from $\varphi(g)\circ k:S^{1}\times F\rightarrow\hat{M}$ is isotopic to the original product structure $k:S^{1}\times F\rightarrow\hat{M}$.

Hence, in regular neighborhoods of each of the connected components, we adjust the product structure $k:S^{1}\times F\rightarrow\hat{M}$ to equal the invariant product structures on the connected components.

It then follows that the respected product structures on each of the boundary tori of $\hat{M}_{1}$ extend within.

We can then apply the inductive hypothesis to assume that $k_{\hat{M}_{1}}:S^{1}\times F_{1}\rightarrow\hat{M}_{1}$ is in fact invariant under the action of $Stab(\hat{M}_{1})$.

We translate this product structure to each $\hat{M}_{i}$ to yield the required invariant product structure.
\end{proof}

\begin{rem} 
We remark here that it is not sufficient simply that there are product structures on the boundary tori that are respected by the action. It is required also that the product structures can be extended within. We give the following example to illustrate this:
\end{rem}

\begin{exmp} 
Let $F$ be an annulus and $k:S^{1}\times F\rightarrow\hat{M}$ be a fibering product structure. Let $G=\mathbb{Z}_{m}$ act on $\hat{M}$ by simply rotating by $\frac{2\pi}{m}$ along the fibers. This action will preserve any fibering product structure (up to isotopy) on each boundary torus. 

Now pick meridians on the first torus to be the loops that are $(0,1)$ curves according to $k:S^{1}\times F\rightarrow\hat{M}$ and meridians on the second torus to be loops that are $(1,1)$ curves according to $k:S^{1}\times F\rightarrow\hat{M}$. These are both left invariant, but there is no product structure on $\hat{M}$ that restricts to these on the boundary.
\end{exmp}

\section{Main Result}
We now prove the main result, which states that given a condition on the obstruction class, any finite, orientation and fiber-preserving action on a closed, compact, and orientable Seifert 3-manifold that fibers over an orientable base space is an extended product action

To prove this, we first state Theorem 2.8.2 of Richard Canary and Darryl McCullough in their book \textit{Homotopy equivalences of 3-manifolds and deformation theory of Kleinian groups} \cite{canary2004homotopy}: 

\begin{thm}
Suppose that each of $(M_{1},\underline{\underline{m_{1}}})$ and $(M_{2},\underline{\underline{m_{2}}})$ is a Seifert-fibered space with nonempty boundary and with fixed admissible fibration, but that neither $(M_{i},\overline{\underline{\underline{m_{i}}}})$ is a solid torus with $\overline{\underline{\underline{m_{1}}}}=\overline{\underline{\underline{\phi}}}$. Let $f:(M_{1},\underline{\underline{m_{1}}})\rightarrow(M_{2},\underline{\underline{m_{2}}})$ be an admissible diffeomorphism, and suppose that for some regular fiber $\gamma$ in $M_{1}$, $f(\gamma)$ is homotopic in $M_{2}$ to a regular fiber. Then $f$ is admissibly isotopic to a fiber-preserving diffeomorphism. If $f$ is already fiber-preserving on some union $U$ of elements of $\underline{\underline{m_{1}}}$, then the isotopy may be chosen to be relative to $U$.
 \end{thm}
 
Here $\underline{\underline{m_{i}}}$ refer to \textit{boundary patterns} of each $M_{i}$. These are finite sets of compact, connected surfaces in $\partial{M_{i}}$, such that the components of the intersections of pairs of elements are arcs or circles, and if any three elements meet, their intersection is a finite collection of points at which three intersection arcs meet. An \textit{admissable fibration} is such that the boundary pattern consists of only tori and annuli, and an \textit{admissable map} is one that sends boundary patterns to boundary patterns.

This then leads us to what we will require:

\begin{lem}
 Let $W$ be a Seifert-fibered torus and let $h:T\rightarrow T$ be a fiber-preserving diffeomorphism with induced homology map  $h_{*}=id$. Then $h:T\rightarrow T$ can be extended to a fiber-preserving diffeomorphism $\overline{h}:T\times I\rightarrow T\times I$ with $\overline{h}(x,1)=(h(x),1), \overline{h}(x,0)=(x,0)$. Here $T\times I$ is fibered as a unique extended fibration. 
 \end{lem}

\begin{proof}
 We note first that an isotopy to the identity exists. We then need only check that such an isotopy can be taken to fiber-preserving. 

As $h_{*}=id$ there exists a diffeomorphism $H:W\times I\rightarrow T$ such that $H(x,1)=h(x)$ and $H(x,0)=x$ with $H_{t}:T\rightarrow T$ a diffeomorphism for each $t\in I$. 

We can then define the diffeomorphism $\tilde{H}:T\times I\rightarrow T\times I$ by $\tilde{H}(x,t)=(H(x,t),t)$. This diffeomorphism is fiber-preserving on the boundary of $T\times I$. 

We then assign $T\times I$ the boundary pattern consisting of the union of its' two boundary tori. Certainly $\tilde{H}$ is an admissible diffeomorphism and moreover it is the identity on one boundary component, so the condition of the image of a fiber being homotopic to a fiber is trivially satisfied. 

It then remains to apply Theorem 5.1 to yield an isotopic map $\bar{h}$ that is fiber-preserving and agrees with $\tilde{H}$ on the boundary. In particular, $\bar{h}(x,1)=\tilde{H}(x,1)=(H(x,1),1)=(h(x),1)$ and $\bar{h}(x,0)=\tilde{H}(x,0)=(H(x,0),0)=(x,0)$.
\end{proof}

It is now possible to restate and prove our main result:

\begin{thm}
\label{thm:associativity}
 Let $M$ be an orientable Seifert 3-manifold that fibers over an orientable base space. Let $\varphi:G\rightarrow Diff_{+}^{fp}(M)$ be a finite group action on $M$ such that the obstruction class can expressed as $$b=\sum_{i=1}^{m}(b_{i}\cdot\#Orb_{\varphi}(\alpha_{i}))$$ for a collection of fibers $\{\alpha_{1},\ldots,\alpha_{m}\}$ and integers $\{b_{1},\ldots,b_{m}\}$. Then $\varphi$ is an extended product action. 
 \end{thm}

\begin{proof}
 We let $M$ be the Seifert $3$-manifold with normalized invariants:$$M=(g,o_{1}|(q_{1},p_{1}),\ldots,(q_{n},p_{n}),(1,b))$$
 
Firstly, without loss of generality, we can assume that the orbits of each $\{\alpha_{1},\ldots,\alpha_{m}\}$ are distinct. If $\alpha_{i},\alpha_{j}$ were in the same orbit, then we note that $b_{i}\cdot\#Orb_{\varphi}(\alpha_{i})+b_{j}\cdot\#Orb_{\varphi}(\alpha_{j})=(b_{i}+b_{j})\cdot\#Orb_{\varphi}(\alpha_{i})$ so that we do not require $\alpha_{j}$ for the property to still hold.

Secondly, we can suppose without loss of generality that the first $t$ of the fibers $\{\alpha_{1},\ldots,\alpha_{t}\}$ are regular and each critical fiber $\{\gamma_{1},\ldots,\gamma_{n}\}$ is in the orbit of one of $\{\alpha_{t+1},\ldots,\alpha_{m}\}$. If one is not, it can be added into the collection with a coefficient of zero. This will not change the sum. 

We start by tasking ourselves with rewriting the Seifert pairings to reflect the assumption that the obstruction class can be expressed as: $$b=\sum_{i=1}^{m}(b_{i}\cdot\#Orb_{\varphi}(\alpha_{i}))$$

Begin by letting: $$A=\sum_{i=1}^{t}\#Orb_{\varphi}(\alpha_{i})$$ and then rewriting the Seifert invariants as: $$M=(g,o_{1}|(q_{1},p_{1}),\ldots,(q_{n},p_{n}),(1,b),(1,0)_{1},\ldots,(1,0)_{A})$$

Here each $(1,0)_{i}$ refers to a regular fiber which is in the orbit of some fiber in the collection $\{\alpha_{1},\ldots,\alpha_{t}\}$. Call this collection of fibers $\{\beta_{1},\ldots,\beta_{A}\}$.

Now let $\{\beta_{A+1},\ldots,\beta_{n+A}\}=\{\gamma_{1},\ldots,\gamma_{n}\}$ and note that $\{\beta_{1},\ldots,\beta_{n+A}\}=Orb_{\varphi}(\{\alpha_{1},\ldots,\alpha_{m}\})$.

Define a function: $h:\{1,\ldots,n+A\}\rightarrow\mathbb{Z}$ by $h(j)=b_{i}\:\textrm{if}\:\beta_{j}\in Orb_{\varphi}(\alpha_{i})$.

Take closed, fibered regular neighborhoods $N(\alpha_{1}),\ldots,N(\alpha_{m})$ and then define: $$X=Orb_{\varphi}(N(\alpha_{1})\cup\ldots\cup N(\alpha_{m}))$$ $$\hat{M}=\overline{M\setminus X}$$

So $X$ is a collection of fibered solid tori and $M$ can be reobtained by some (fiber-preserving) gluing map $d:\partial X\rightarrow\partial\hat{M}$. This gluing map corresponds to the presentation: $$M=(g,o_{1}|(q_{1},p_{1}+h(1)q_{1}),\ldots,(q_{n},p_{n}+h(n)q_{n}),(1,h(n+1)),\ldots,(1,h(n+A))$$

This is possible by Corollary 2.2 as $$\sum_{j=1}^{n+A}h(j)=\sum_{i=1}^{m}b_{i}\cdot\#Orb_{\varphi}(\alpha_{i})=b$$

For convenience, denote: \begin{multline*}
(g,o_{1}|(q_{1},p_{1}+h(1)q_{1}),\ldots,(q_{n},p_{n}+h(n)q_{n}),(1,h(n+1)),\ldots,(1,h(n+A)) \\ =(g,o_{1}|(q'_{1},p'_{1}),\ldots,(q'_{n},p'_{n}),(q'_{n+1},p'_{n+1}),\ldots,(q'_{n+A},p'_{n+A}))
\end{multline*}

We then proceed with this equivalent representation.

From Section 2, this gives us a fibering product structure ${\hat{M}}:S^{1}\times F\rightarrow\hat{M}$ and a product structure $k_{X}:S^{1}\times(D_{1}\cup\ldots\cup D_{n+A})\rightarrow X$ so that according to it, each $V_{i}$ in $X$ has a normalized fibration. 

We then have that $(d|_{\partial V_{i}})_{*}=\left[\begin{array}{cc}
x'_{i} & p'_{i}\\
y'_{i} & q'_{i}
\end{array}\right]_{k_{T_{i}}}^{k_{\partial V_{i}}}=\left[\begin{array}{cc}
x'_{i} & p_{i}+h(i)q_{i}\\
y'_{i} & q_{i}
\end{array}\right]_{k_{T_{i}}}^{k_{\partial V_{i}}}$ for the nontrivially fibered solid tori according to these product structures. 

The fibrations on each $V_{i}$ is a $(-q_{i},y'_{i})$ fibration and the action can only send some $V_{i}$ to a $V_{j}$ if they have the same fibration. Hence $(-q_{i},y'_{i})=(-q_{j},y'_{j})$. 

We now show that the action can only send some $V_{i}$ to a $V_{j}$ if they have the same associated fillings.

Beginning with $x'_{i}q_{i}-y'_{i}(p_{i}+h(i)q_{i})=-1$ and $x'_{j}q_{i}-y'_{i}(p_{j}+h(i)q_{i})=-1$ we yield: \begin{align*}
x'_{i}q_{i}(p_{j}+h(i)q_{i})-y'_{i}p_{i}(p_{j}+h(i)q_{i})&=-(p_{j}+h(i)q_{i})\\
x'_{j}q_{i}(p_{i}+h(i)q_{i})-y'_{i}p_{j}(p_{i}+h(i)q_{i})&=-(p_{i}+h(i)q_{i})
\end{align*}
So that $q_{i}(x'_{i}(p_{j}+h(i)q_{i})-x'_{j}(p_{i}+h(i)q_{i}))=p_{i}-p_{j}$.

However, $-q_{i}<p_{i}-p_{j}<q_{i}$, hence $-1<(x'_{i}(p_{j}+h(i)q_{i})-x'_{j}(p_{i}+h(i)q_{i}))<1$, and so $x'_{i}(p_{j}+h(i)q_{i})=x'_{j}(p_{i}+h(i)q_{i})$.

But $x'_{i},(p_{i}+h(i)q_{i})$ are coprime and so are $x'_{j},(p_{j}+h(i)q_{i})$, hence $x'_{i}=x'_{j}$ and $(p_{i}+h(i)q_{i})=(p_{j}+h(i)q_{i})$.

So finally $p_{i}=p_{j}$, as well as $p'_{i}=p'_{j}$ and we can henceforth assume that if the action sends some $V_{i}$ to a $V_{j}$, then the fillings must be the same. Note that this is true also for the fillings of trivially fibered tori by construction.

We here consider $\hat{M}$. It is a Seifert-fibered $3$-manifold with boundary such that there is a fiber-preserving restricted action given by:$$\hat{\varphi}:G\rightarrow Diff_{+}^{fp}(\hat{M})$$ $$\hat{\varphi}(g)=\varphi(g)|_{\hat{M}}$$

We now proceed to show that there is a product structure on $\hat{M}$ such that $\hat{\varphi}$ respects the restricted product structures on the boundary tori. We do so to employ Theorem 4.5. 

Take $T_{i}$ arbitrarily and consider the action given by $\hat{\varphi}(g)|_{T_{i}}$ for each $g\in Stab(T_{i})$.

By restricting $k_{\hat{M}}:S^{1}\times F\rightarrow\hat{M}$ and $k_{X}:S^{1}\times(D_{1}\cup\ldots\cup D_{n+A})\rightarrow X$ as in Section 2 to $k_{T_{i}}:S^{1}\times S^{1}\rightarrow T_{i}$ and $k_{\partial V_{i}}:S^{1}\times S^{1}\rightarrow\partial V_{i}$ we have the following homological diagram: $$\begin{array}{ccccc}
 &  & (d|_{\partial V_{i}})_{*}\\
 & H_{1}(T_{i}) & \leftarrow & H_{1}(\partial V_{i})\\
(\hat{\varphi}(g)|_{T_{i}})_{*} & \downarrow &  & \downarrow & (d|_{\partial V_{i}}^{-1}\circ\hat{\varphi}(g)|_{T_{i}}\circ d|_{\partial V_{i}})_{*}\\
 & H_{1}(T_{i}) & \leftarrow & H_{1}(\partial V_{i})\\
 &  & (d|_{\partial V_{i}})_{*}
\end{array}$$

As the action extends into $V_{i}$ and is finite, we must have that $(d|_{\partial V_{i}}^{-1}\circ\hat{\varphi}(g)|_{T_{i}}\circ d|_{\partial V_{i}})_{*}=\pm id$. Hence $(\hat{\varphi}(g)|_{T_{i}})_{*}=\pm id$ for all $g\in Stab(T_{i})$.

We can then apply Lemma 4.2 to get $f_{i}:T_{i}\rightarrow T_{i}$ such that $f_{i}$ is fiber-preserving, isotopic to the identity, and $k_{T_{i}}^{-1}\circ f_{i}^{-1}\circ\hat{\varphi}(g)|_{T_{i}}\circ f_{i}\circ k_{T_{i}}$ is a product map for each $g\in Stab(T_{i})$.

Now pick $g_{j}\in G$ for each $T_{j}\in Orb(T_{i})$ such that $\hat{\varphi}(g_{j})(T_{i})=T_{j}$. 

We translate the conjugating map $f_{i}:T_{i}\rightarrow T_{i}$ to each $T_{j}\in Orb(T_{i})$ by defining $f_{j}=\hat{\varphi}(g_{j})|_{T_{i}}\circ f_{i}\circ k_{T_{i}}\circ h_{j}\circ k_{T_{j}}{}^{-1}$ where:$$h_{j}(u,v)=\left\{ \begin{array}{ccc}
(u,v) & & \thinspace\textrm{               if }\hat{\varphi}(g_{j})\textrm{ preserves the orientation of the fibers}\\
(u^{-1},v^{-1}) &  & \textrm{if }\hat{\varphi}(g_{j})\textrm{ reverses the orientation of the fibers}
\end{array}\right.$$

Each $f_{j}$ is certainly fiber-preserving, but we must check that they are isotopic to the identity. 

To do so, note that we have the diagram:$$\begin{array}{ccccc}
 &  & \left[\begin{array}{cc}
x'_{i} & p'_{i}\\
y'_{i} & q'_{i}
\end{array}\right]_{k_{T_{i}}}^{k_{\partial V_{i}}}\\
 & H_{1}(T_{i}) & \leftarrow & H_{1}(\partial V_{i})\\
\hat{\varphi}(g_{j})_{*} & \downarrow &  & \downarrow & \pm id\\
 & H_{1}(T_{j}) & \leftarrow & H_{1}(\partial V_{j})\\
 &  & \left[\begin{array}{cc}
x'_{i} & p'_{i}\\
y'_{i} & q'_{i}
\end{array}\right]_{k_{T_{j}}}^{k_{\partial V_{j}}}
\end{array}$$

So that necessarily $\hat{\varphi}(g_{j})_{*}=\pm id$ depending on whether the orientation on the fibers are reversed or not. Consequently, $f_{j}$ is isotopic to the identity.

Then for any $g\in G$, $g=g_{j_{2}}hg_{j_{1}}^{-1}$, for some $h\in Stab(T_{i})$ and some $T_{j_{1}},T_{j_{2}}\in Orb(T_{i})$. We calculate: $k_{T_{j_{2}}}^{-1}\circ f_{j_{2}}^{-1}\circ\hat{\varphi}(g)|_{T_{j_{1}}}\circ f_{j_{1}}\circ k_{T_{j_{1}}}=	\,h_{j_{2}}^{-1}\circ(k_{T_{i}}^{-1}\circ f_{i}^{-1}\circ\hat{\varphi}(h)|_{T_{i}}\circ f_{i}\circ k_{T_{i}})\circ h_{j_{1}}$. So that $k_{T_{j_{2}}}^{-1}\circ f_{j_{2}}^{-1}\circ\hat{\varphi}(g)|_{T_{j_{1}}}\circ f_{j_{1}}\circ k_{T_{j_{1}}}$ is also a product map, and the product structures $f_{j}\circ k_{T_{j}}:S^{1}\times S^{1}\rightarrow T_{j}$ for $T_{j}\in Orb(T_{i})$ are invariant under $\hat{\varphi}$. 

We can now do this for each of the distinct orbits of boundary tori.

As each $f_{j}$ is isotopic to the identity and fiber-preserving, we can employ Lemma 5.2 to define $f\in Diff_{+}^{fp}(\hat{M})$ so that $f|_{T_{j}}=f_{j}$ and $f$ is the identity outside of a regular neighborhood of each boundary torus. $f$ is necessarily isotopic to the identity.

So now, the product structure $f\circ k_{\hat{M}}:S^{1}\times F\rightarrow\hat{M}$ is such that $f\circ k_{T_{j}}:S^{1}\times S^{1}\rightarrow T_{j}$ for each $T_{j}$ are respected under $\hat{\varphi}$ and moreover is isotopic to $k_{\hat{M}}$. 

Then we have what we require to employ Theorem 4.5: a product structure on $\hat{M}$ such that $\hat{\varphi}$ respects the restricted product structures on the boundary tori. So we yield a product structure $k'_{\hat{M}}:S^{1}\times F\rightarrow\hat{M}$ such that each $k_{\hat{M}}^{\prime-1}\circ\hat{\varphi}(g)\circ k'_{\hat{M}}$ is a product map. We can assume that each component of $k_{\hat{M}}^{\prime-1}\circ\hat{\varphi}(g)\circ k'_{\hat{M}}$ is an isometry under some appropriate metrics on $S^{1}$ and $F$.

Therefore, we must have that on each boundary component $T_{i}$: $$(k_{T_{\beta(g)(i)}}^{\prime-1}\circ\hat{\varphi}(g)\circ k'_{T_{i}})(u,v)=(\theta_{1}(g)u^{\alpha_{1}(g)},\theta_{2}(i,g)v^{\alpha_{2}(g)})$$

But now $\alpha_{1}(g)=\alpha_{2}(g)$ as the action is orientation-preserving. 

It remains to show that we can pick a product structure on $X$ that is left invariant. We know that there is a product structure $k'_{X}:S^{1}\times(D_{1}\cup\ldots\cup D_{l})\rightarrow X$ so that according to the product structure $k'_{\hat{M}}:S^{1}\times F\rightarrow\hat{M}$ we have: $$(k_{T_{i}}^{\prime-1}\circ d|_{T_{i}}\circ k\prime_{\partial V_{i}})(u,v)=(u^{x'_{i}}v^{p'_{i}},u^{y'_{i}}v^{q'_{i}})$$

If we let $\varphi_{X}$ be the action restricted to $X$, we have that according to this product structure, the action on the boundary of $X$ looks like:$$(k_{\partial V_{\beta(g)(i)}}^{\prime-1}\circ\varphi_{X}(g)\circ k\prime_{\partial V_{i}})(u,v)=(\theta_{1}(g)^{-q_{i}}\theta_{2}(i,g)^{p_{i}}u^{\alpha(g)},\theta_{1}(g)^{y_{i}}\theta_{2}(i,g)^{-x_{i}}v^{\alpha(g)})$$

That is, it respects the restricted product structures. Hence we can consider $Stab(V_{i})$ for each $V_{i}$ to apply Theorem 4.5 and translate in a similar way to above and in the proof of Theorem 4.5.

This completes the proof.
\end{proof}

As a result of Theorem 5.3 we yield the following: 

\begin{cor} Let $M$ be an orientable Seifert $3$-manifold that fibers over an orientable base space. Let $\varphi:G\rightarrow Diff_{+}^{fp}(M)$ be a finite group action on $M$ such that a fiber is left invariant. Then $\varphi$ is an extended product action. 
\end{cor}

\begin{proof}
Let $\alpha$ be the fiber left invariant. Then $\#Orb_{\varphi}(\alpha)=1$ and so $b=b\cdot\#Orb_{\varphi}(\alpha)$.
\end{proof}

\begin{cor}
 Let $M$ be an orientable Seifert $3$-manifold that fibers over an orientable base space with only one cone point of order $q$. Let $\varphi:G\rightarrow Diff_{+}^{fp}(M)$ be a finite group action on $M$. Then $\varphi$ is an extended product action.
 \end{cor}

\begin{proof}
 Let $\alpha$ be the fiber that refers to the cone point of order $q$. Then $\#Orb_{\varphi}(\alpha)=1$ and so $b=b\cdot\#Orb_{\varphi}(\alpha)$.
\end{proof}

\begin{cor} Let $M$ be an orientable Seifert $3$-manifold that fibers over an orientable base space. Let $\varphi:G\rightarrow Diff_{+}^{fp}(M)$ be a finite group action on $M$ so that there are two fibers $\alpha,\beta$ with $\#Orb_{\varphi}(\alpha),\#Orb_{\varphi}(\beta)$ coprime. Then $\varphi$ can be derived via the construction set out in Section 3.
\end{cor}

\begin{proof}
There exists $x,y\in\mathbb{Z}$ such that $x\cdot\#Orb_{\varphi}(\alpha)+y\cdot\#Orb_{\varphi}(\beta)=1$ and so $b=bx\cdot\#Orb_{\varphi}(\alpha)+by\cdot\#Orb_{\varphi}(\beta)$.
\end{proof}

These corollaries give some simple situations under which the conditions of Theorem 5.3 are satisfied. We use the following section to present some concrete examples of the use of these corollaries.

\section{Examples: Part One}

We give some concrete examples in this section with two specific $3$-manifolds that serve to highlight the corollaries above. We will revisit examples in section 8 after we have further analyzed the condition on the obstruction class.

\begin{exmp}
Take any Seifert manifold with a critical fiber of order different from all others. In particular, we can choose a lens space $M=(0,o_{1}|(3,2))$. This lens space has only one critical fiber of order 3. Drilling out the critical fiber leaves a trivially fibered solid torus.

We can then employ Corollary 5.5 to see that any action on $M$ will be an extended product action of a product action on $S^{1} \times D$. These actions have been well considered in particular in \cite{kalliongis1991symmetries} and are generated by rotations in each component along with the aforementioned "spin" - a reflection in both components. 
\end{exmp}

\begin{exmp}

We consider a Seifert manifold $M$ which fibers over an orientable base space $B$ which has the cone points $2,2,3,3,3$. Now any action on $B$ would necessarily only be able to exchange the two cone points of order $2$ and permute the cone points of order $3$. Hence a critical fiber $\alpha$ referring to one of the cone points of order $2$, must have that $\#Orb_{\varphi}(\alpha)$ is $1$ or $2$. Similarly, there is a critical fiber $\beta$ referring to one of the cone points of order $3$, that must have either $\#Orb_{\varphi}(\beta)$ as $1$ or $3$. If either $\#Orb_{\varphi}(\alpha)$ or $\#Orb_{\varphi}(\beta)$ is $1$, then we can apply Corollary 5.4. If $\#Orb_{\varphi}(\alpha)=2$ and $\#Orb_{\varphi}(\beta)=3$, then we can apply Corollary 5.6.

In all cases any finite, orientation and fiber-preserving action on $M$ must be derived via the construction set out in Section 3. This is regardless of the obstruction class.

We give a specific manifold to illuminate this. Let $M=(0,o_{1}|(2,1),(2,1),(3,1),(3,1),(3,1))$. This is in particular a hyperbolic manifold as the orbifold Euler number of the base space $B=S^{2}(2,2,3,3,3)$ is $\chi(B)=2-(1-\frac{1}{2})-(1-\frac{1}{2})-(1-\frac{1}{3})-(1-\frac{1}{3})-(1-\frac{1}{3})=-1<0$.

So now drilling out these critical fibers will leave $\hat{M}\cong S^{1} \times F$ where $F$ is the closure of $S^{2}$ with 5 discs removed. Any action on $F$ can only exchange two of the boundary components and permute the remaining three. Referring to John Kalliongis and Ryo Ohashi's paper \textit{Finite actions on the $2$-sphere, the projective plane and $I$-bundles over the projective plane} \cite{kalliongis2018}, we learn that we need the group to be a subgroup of a group of the form $Dih(\mathbb{Z}_{3})$ generated by an order three rotation that fixes two boundary components and either an order two rotation or a relection.
\end{exmp}

\section{Obstruction Condition}

If $\varphi:G\rightarrow Diff_{+}^{fp}(M)$ is a finite group action, we henceforth call satisfaction of $$b=\sum_{i=1}^{s}(b_{i}\cdot\#Orb_{\varphi}(\alpha_{i}))$$ for some fibers $\{\alpha_{1},\ldots,\alpha_{s}\}$ and integers $\{b_{1},\ldots,b_{s}\}$, \textit{satisfying the obstruction condition}. 

\begin{rem}
We note that the obstruction condition is not always satisfied. We give a specific example in the following section.
\end{rem}

We now proceed to refine the obstruction condition. First, two lemmas are established and then a proposition which provides a convenient equivalent statement for the obstruction condition that can be used to apply our results.

\begin{lem}
Let $\varphi:G\rightarrow Diff(S)$ be a finite group action on a surface $S$. Suppose that the orbifold $S/\varphi$ has data set $(n_{1},\ldots,n_{k};m_{1},\ldots,m_{l})$. Then the possible orbit numbers under $\varphi$ are $|G|/n_{1},\ldots,|G|/n_{k},|G|/2m_{1},\ldots,|G|/2m_{l}$ and $|G|$.
\end{lem}

\begin{proof}
 $S$ is an order $|G|$ orbifold cover of $S/\varphi$. Therefore any regular point of $S/\varphi$ lifts to $|G|$ points of $S$, any of these points have orbit number $|G|$. Any neighborhood of a cone point of order $n_{i}$ is covered by a collection of discs in $S$, each disc is an $n_{i}$-fold cover of the neighborhood. Hence the number of discs that cover the neighborhood is $\frac{|G|}{n_{i}}$. Thus the center of each disc has orbit number $\frac{|G|}{n_{i}}$.

Any neighborhood of a corner reflector of order $m_{i}$ is covered by a collection of discs in $S$, each disc is an $2m_{i}$-fold cover of the neighborhood. Hence the number of discs that cover the neighborhood is $\frac{|G|}{2m_{i}}$. Thus the center of each disc has orbit number $\frac{|G|}{2m_{i}}$.
\end{proof}

\begin{lem}
 Let $n_{1},\ldots,n_{k}$ be factors of $N$. Then $\frac{N}{lcm(n_{1},\ldots,n_{k})}=gcd(\frac{N}{n_{1}},\ldots,\frac{N}{n_{k}})$. 
\end{lem}

\begin{proof}
 We work by induction. For the initial case we use the result that $gcd(x,y)lcm(x,y)=xy$ for any integers $x,y$. This implies: \begin{align*}
gcd(\frac{N}{n_{1}},\frac{N}{n_{2}})lcm(n_{1},n_{2})&=\frac{N^{2}lcm(n_{1},n_{2})}{n_{1}n_{2}lcm(\frac{N}{n_{1}},\frac{N}{n_{2}})}\\ &=\frac{N^{2}lcm(n_{1},n_{2})}{lcm(n_{2}N,n_{1}N)}\\&=\frac{N^{2}lcm(n_{1},n_{2})}{Nlcm(n_{2},n_{1})}\\ &=N
\end{align*}
For the inductive step, we work in a similar fashion: \begin{align*}
    gcd(\frac{N}{n_{1}},\ldots,\frac{N}{n_{k}})	&=gcd(gcd(\frac{N}{n_{1}},\ldots,\frac{N}{n_{k-1}}),\frac{N}{n_{k}})\\
	&=gcd(\frac{N}{lcm(n_{1}\ldots,n_{k-1})},\frac{N}{n_{k}})\\
	&=\frac{N^{2}}{n_{k}lcm(n_{1},\ldots,n_{k-1})lcm(\frac{N}{lcm(n_{1},\ldots,n_{k-1})},\frac{N}{n_{k}})}\\
	&=\frac{N^{2}}{lcm(Nn_{k},Nlcm(n_{1},\ldots n_{k-1})}\\
	&=\frac{N}{lcm(n_{1},\ldots,n_{k})}
\end{align*}
\end{proof}

\begin{prop}
 Let $\varphi:G\rightarrow Diff_{+}^{fp}(M)$ be a finite group action and $\varphi_{B_{U}}:G_{B_{U}}\rightarrow Diff(B_{U})$ the induced action on the underlying space of the base space $B$ which has branching data $(n_{1},\ldots,n_{k};m_{1},\ldots,m_{l})$. Then $\varphi:G\rightarrow Diff_{+}^{fp}(M)$ satisfies the obstruction condition if and only if $\frac{|G_{B_{U}}|}{lcm(n_{1},\ldots,n_{k},2m_{1},\ldots,2m_{l})}$ divides $b$.
 \end{prop}

\begin{proof}
 We first note that there exist fibers $\{\alpha_{1},\ldots,\alpha_{s}\}$ and integers $\{b_{1},\ldots,b_{s}\}$ such that: $$b=\sum_{i=1}^{s}(b_{i}\cdot\#Orb_{\varphi}(\alpha_{i}))$$ if and only if there exist points $\{x_{1},\ldots,x_{s}\}\subset B_{U}$ and integers $\{b_{1},\ldots,b_{s}\}$ such that: $$b=\sum_{i=1}^{s}(b_{i}\cdot\#Orb_{\varphi_{B_{U}}}(x_{i}))$$

We begin with the if statement. So by Lemma 7.2, $\frac{|G_{B_{U}}|}{lcm(n_{1},\ldots,n_{k},2m_{1},\ldots,2m_{l})}=gcd(\frac{|G_{B_{U}}|}{n_{1}},\ldots,\frac{|G_{B_{U}}|}{n_{k}},\frac{|G_{B_{U}}|}{2m_{1}},\ldots,\frac{|G_{B_{U}}|}{2m_{l}})$ divides $b$.

Hence there exist $\{b_{1},\ldots,b_{k+l}\}$ such that: $$b=\sum_{i=1}^{k}b_{i}\cdot\frac{|G_{B_{U}}|}{n_{i}}+\sum_{i=1}^{l}b_{i}\cdot\frac{|G_{B_{U}}|}{2m_{i}}$$ by Euclid's algorithm.

So by Lemma 7.1, there are $\{x_{1},\ldots,x_{k},x_{k+1},\ldots,x_{k+l}\}\subset B_{U}$ such that $\#Orb_{\varphi_{B_{U}}}(x_{i})=\frac{|G_{B_{U}}|}{n_{i}}$ and $\#Orb_{\varphi_{B_{U}}}(x_{i})=\frac{|G_{B_{U}}|}{2m_{i}}$. Thus: $$b=\sum_{i=1}^{k+l}b_{i}\cdot\#Orb_{\varphi_{B_{U}}}(x_{i})$$

For the only if, suppose that there exist points $\{x_{1},\ldots,x_{s}\}\subset B_{U}$ and integers $\{b_{1},\ldots,b_{s}\}$ such that: $$b=\sum_{i=1}^{s}(b_{i}\cdot\#Orb_{\varphi_{B_{U}}}(x_{i}))$$ 

Without loss of generality, we can assume that the orbit numbers of all the $x_{i}$ are different, that $s=k+l$ (set $b_{i}=0$ if necessary), and that the branching data of each $x_{i}$ is $n_{i}$ for $i=1,\ldots,k$ and $2m_{i}$ for $i=k+1,\ldots,l$. 

Hence, by Lemma 7.1: $$b=\sum_{i=1}^{k}b_{i}\cdot\frac{|G_{B_{U}}|}{n_{i}}+\sum_{i=1}^{l}b_{i}\cdot\frac{|G_{B_{U}}|}{2m_{i}}$$ and so:$$gcd(\frac{|G_{B_{U}}|}{n_{1}},\ldots,\frac{|G_{B_{U}}|}{n_{k}},\frac{|G_{B_{U}}|}{2m_{1}},\ldots,\frac{|G_{B_{U}}|}{2m_{l}})$$ divides $b$. 

Finally, by Lemma 7.1, $\frac{|G_{B_{U}}|}{lcm(n_{1},\ldots,n_{k},2m_{1},\ldots,2m_{l})}$ divides $b$.
\end{proof}

This result then allows us to quickly establish whether the obstruction condition is satisfied based on the order of the induced action on the base space and the least common multiple of the data from the orbifold quotient of the induced action. This is a convenient way to establish results based on possible quotient types.

\section{Examples: Part Two}

We begin this second set of examples with an action that does not satisfy the obstruction condition.
\begin{exmp}
Construct by a Seifert $3$-manifold $M$ fibering over an even genus $g$ surface with no critical fibers and odd obstruction $b$ by taking two trivially fibered manifolds $M_{1}=S^{1}\times F_{1}$ and $M_{2}=S^{1}\times F_{2}$ where $F_{1},F_{2}$ are genus $\frac{g}{2}$ surfaces with a disc removed, and then gluing according to the map $d(u_{1},v_{1})=(u_{2}^{-1}v_{2}^{b},v_{2})$ between boundary tori.

Define the rotation $rot_{2}:F_{i}\rightarrow F_{i}$ to be an order $2$ rotation that leaves the boundary invariant.

Then consequently define an orientation-preserving, finite and fiber-preserving action on $M_{1}$ and $M_{2}$ by $f_{i}:S^{1}\times F_{i}\rightarrow S^{1} \times F_{i}$ with: $$f_{1}(u_{1},x_{1})=(u_{1},rot_{2}(x_{1})), f_{1}(u_{2},x_{2})=(-u_{2},rot_{2}(x_{2}))$$ $$f_{2}(u_{1},x_{1})=(u_{2},x_{2}), f_{2}(u_{2},x_{2})=(u_{1},x_{1})$$

It can be checked that these agree over the gluing torus.

So then the projected action on the genus $g$ surface is a $\mathbb{Z}_{2}\times\mathbb{Z}_{2}$-action and all orbit numbers are even. Hence, it cannot be that: $$b=\sum_{i=1}^{s}(b_{i}\cdot\#Orb_{\varphi}(\alpha_{i}))$$
\end{exmp}

We now adjust this example to some specific manifolds that have even obstruction class.

\begin{exmp}
We take the lens space given by $M=(0,o_{1}|(3,2),(3,2),(1,2))$. We note that certainly the two critical fibers can be exchanged and in fact the action defined as in Example 8.1 will do this. However, in this case the obstruction class is even and so the obstruction condition will be satisfied. In particular, we can see the rearrangement of the Seifert pairings that would allow this as:$$M=(0,o_{1}|(3,2+3),(3,2+3),(1,2-2))=(0,o_{1}|(3,5),(3,5))$$
\end{exmp}

In a future paper, all Elliptic manifolds will be considered and the results obtained here will serve to derive all possible finite fiber-preserving group actions subject to the obstruction condition.

\section{Group Structures}

We now establish the possible structures of the groups that can act fiber- and orientation-preservingly on a Seifert manifold (satisfying the obstruction condition). 

We firstly prove the following:

\begin{prop}
 Suppose that $\varphi:G\rightarrow Diff(S^{1})\times Diff(F)$ is a finite group action with $\varphi(g)(u,x)=(\varphi_{S^{1}}(g)(u),\varphi_{F}(g)(x))$ such that $\varphi_{S^{1}}(g)$ is orientation-preserving if and only if $\varphi_{F}(g)$ is orientation-preserving. Suppose that there exists $g_{-}\in G$ such that $\varphi_{S^{1}}(g_{-})$ is orientation-reversing and $g_{-}^{2}=1$. Then $G$ is isomorphic to a subgroup of a semidirect product of $\mathbb{Z}_{n}\times\varphi_{F}(G)_{+}$ and $\mathbb{Z}_{2}$.
\end{prop}

\begin{proof}
 First let $\varphi(G)^{fop}$ be the subgroup of $\varphi(G)$ where each element is orientation-preserving on both components. 

Now consider the structure of $\varphi(G)^{fop}$ and note that $\varphi(G)^{fop}$ is a finite subgroup of $\varphi_{S^{1}}(G)_{+}\times\varphi_{F}(G)_{+}$. We have that $\varphi_{S^{1}}(G)_{+}\cong\mathbb{Z}_{n}$ for some $n$ and so $\varphi(G)^{fop}$ is a finite subgroup of $\mathbb{Z}_{n}\times\varphi_{F}(G)_{+}$.

We then consider the short-exact sequence $1\rightarrow\varphi(G)^{fop}\rightarrow\varphi(G)\rightarrow\mathbb{\mathbb{Z}}_{2}\rightarrow1$.

This splits if there is an element in $\varphi(G)$ of order $2$ that is not in $\varphi(G)^{fop}$. By assumption, $\varphi(g_{-})$ is such an element. The result then follows.
\end{proof}

This result then leads to the following corollaries:

\begin{cor}
Let $M$ be an orientable Seifert 3-manifold that fibers over an orientable base space. Let $\varphi:G\rightarrow Diff_{+}^{fp}(M)$ be a finite group action on $M$ that satisfies the obstruction condition. Suppose that the action preserves the orientation of the fibers Then $G$ is isomorphic to a subgroup of $\mathbb{Z}_{n}\times H$ where $H$ is a group that acts orientation-preservingly on the base space.
\end{cor}

\begin{cor}
Let $M$ be an orientable Seifert 3-manifold that fibers over an orientable base space. Let $\varphi:G\rightarrow Diff_{+}^{fp}(M)$ be a finite group action on $M$ that satisfies the obstruction condition. Suppose that there exists $g_{-}\in G$ such that $\varphi(g_{-})$ reverses the orientation of the fibers $g_{-}^{2}=1$. Then $G$ is isomorphic to a subgroup of a semidirect product of $\mathbb{Z}_{n}\times H$ and $\mathbb{Z}_{2}$ where $H$ is a group that acts orientation-preservingly on the base space.
\end{cor}

These results give us the opportunity to reduce our question that we started the paper with to a question of which finite groups act on a surface. At least in the case of low genus surface, this is a known quantity. 

\section{Summary}
We have shown that provided that the obstruction condition is satisfied, then a finite, fiber- and orientation-preserving action can be constructed via our method. The final section above gives some form to the kinds of finite groups that act this way. We note that there is the restriction that $G$ contains an order $2$ element that reverses the orientation of the fibers and therefore reverses the orientation on the base space. In the particular case of the base space being $S^{2}$ this is not a restriction as any finite group that acts is a subgroup of a finite group that has this property. For clarification of this see again \cite{kalliongis2018}.

In particular, we will establish in a future paper that the finite groups that act fiber- and orientation-preservingly on Seifert manifolds fibering over $S^{2}$ (and satisfiying the obstruction condition) are of the form $(\mathbb{Z}_{n}\times H)\circ_{-1}\mathbb{Z}_{2}$ where $\mathbb{Z}_{2}$ acts by anticommuting with each element of $\mathbb{Z}_{n}\times H$ and $H$ is one of either the trivial group,$\mathbb{Z}_{n}$, $Dih(\mathbb{Z}_{n})$, $A_{4}$, $S_{4}$, or $A_{5}$.

\bibliographystyle{unsrt}
\bibliography{references.bib}

\end{document}